\documentclass[12pt]{amsart}
\usepackage{a4wide}
\usepackage{url}
\usepackage{amssymb}
\usepackage{verbatim}
\newcommand{\abs}[1]{\left\lvert #1 \right\rvert}  



\DeclareMathOperator{\StochMeas}{\mathrm{St}}                     
\DeclareMathOperator{\ProdMeas}{\mathrm{Pr}}                     

\renewcommand{\phi}{\varphi{}}

\newcommand{\IC}{\mathbf{C}}                     
\newcommand{\IN}{\mathbf{N}}                     
\newcommand{\IR}{\mathbf{R}}                     

\newcommand{\SL}{\mathrm{SL}}                     
\newcommand{\GL}{\mathrm{GL}}                     
\newcommand{\RR}{R}                     
\DeclareMathOperator{\NC}{{\mathrm NC}}                     
\newcommand{\Pol}{\mathbf{P}}                     
\newcommand{\symm}{\phi}                     
\newcommand{\alg}[1]{\mathcal{#1}}                     
\newcommand{\SG}{\mathfrak{S}}                     
\newcommand{\NCsymbol}[2]{\left[#1\ #2\right]} 

\makeatletter
  \DeclareRobustCommand\em
         {\@nomath\em \ifdim \fontdimen\@ne\font >\z@
                       \upshape \else \slshape \fi}
\makeatother

\newtheoremstyle{mythm}
  {9pt}
  {9pt}
  {\slshape}
  {0pt}
  {\bfseries}
  {.}
  { }
  {\thmname{#1} \thmnumber{#2}\thmnote{(#3)}}

\theoremstyle{mythm}

\newtheorem{Theorem}{Theorem}[section]
\newtheorem{Lemma}[Theorem]{Lemma}

\newtheorem{Corollary}[Theorem]{Corollary}

\theoremstyle{definition} 
\newtheorem{Definition}[Theorem]{Definition}
\newtheorem{Remark}[Theorem]{Remark}

\newtheorem{Example}[Theorem]{Example}

\numberwithin{equation}{section}

\usepackage[inline]{asymptote}

\begin{asydef}
  real textwidth=10cm;
  real h0=0.5;
  real h1=0.8;  
  real h2=1.1;  
  real h3=1.4;  
\end{asydef}

\begin{document}
\title[Noncommutative Invariants]{A Noncrossing Basis for Noncommutative Invariants of $SL(2,\IC)$}
\author{Franz Lehner}
\begin{abstract}
  Noncommutative invariant theory is a generalization of
  the classical invariant theory of the action of $SL(2,\IC)$
  on binary forms. 
  The dimensions of the spaces of invariant noncommutative polynomials
  coincide with the numbers of certain noncrossing partitions.
  We give an elementary combinatorial explanation of this fact by constructing
  a noncrossing basis of the homogeneous components.
  Using the theory free stochastic measures this provides a combinatorial
  proof of the Molien-Weyl formula in this setting.
\end{abstract}
\date{\today}
\maketitle{}

Invariant theory has played a major role in 19th century mathematics.
It has seen a revival in the last decades and one of the recent
generalizations is noncommutative invariant theory.
The study of noncommutative invariants of $\SL(n,\IC)$ 
has been initiated by Almkvist, Dicks, Formanek and Kharchenko
\cite{Kharchenko:1978:algebras,DicksFormanek:1982:Poincare,AlmkvistDicksFormanek:1985:Hilbert},
see \cite{Almkvist:1990:commutative} for a survey.
An approach using Young tableaux was realized by Teranishi 
\cite{Teranishi:1988:noncommutative} and the symbolic method
was adapted from the classical  to the noncommutative setting by 
Tambour~\cite{Tambour:1991:noncommutative}.
The latter provides the ground on which we
establish a natural basis of the noncommutative
invariants which is in bijection with certain noncrossing 
partitions.
It arose after computer experiments and subsequent consulting of
Sloane's database~\cite{Sloane:Encyclopedia}. 
This bijection is applied to provide a combinatorial proof of
the Molien-Weyl integral formula for the Hilbert-Poincar\'e series
in this setting, using free cumulants and free stochastic measures.

This note is organized as follows.
In Section~\ref{sec:outline} we give a short survey
of invariant theory and the statement of the problem.
In Section~\ref{sec:freeprobability} we review a few facts
from free probability theory and noncrossing partitions.
In Section~\ref{sec:noncrossing} we explain the symbolic method
and construct the noncrossing basis announced in the title.
In Section~\ref{sec:freestochastic} we review the necessary
combinatorial aspects of free stochastic measures
and conclude by a proof of the Molien-Weyl formula
using the newly found noncrossing basis.

\section{An Outline of Invariant Theory}
\label{sec:outline}
\subsection{Introduction}

Let $X$ be a set and $G$ a group acting on $X$ from the left.
Consider a class $\alg{A}$ of functions
$f:X\to Y$, usually an algebra or at least a vector space,
on which the induced action of $G$
\begin{equation}
  \label{eq:inducedaction}
  (g\cdot f)(x) = f(g^{-1}x)
\end{equation}
makes sense.
The objects of \emph{invariant theory} are the fixed point sets
$$
\alg{A}^G = \{ f\in \alg{A} : (g\cdot f)=f \;\forall g\in G\}
$$
of such actions.
\begin{Example}
  A favourite example is provided by quadratic polynomials
  and the group $G$ of translations of the real axis $\IR$:
  $$
  g_s:x\mapsto x+s
  .
  $$
  Denote $X=\IR_2[x]=\{a_0+a_1x+a_2x^2 : a_0,a_1,a_2\in \IR\}$ 
  the space of polynomials of degree~$2$
  with the action of $G$
  $$
  (g_s\cdot p)(x) = p(x-s) = a_0-a_1s+a_2s^2+(a_1-2a_2s)x+a_2x^2
  .
  $$
  Now one may ask which properties of a polynomial $p=a_0+a_1x+a_2x^2$
  do not change under translation.
  One significant parameter is the number of distinct real roots of a quadratic
  polynomial,   and the three possibilities are distinguished by
  the sign of the \emph{discriminant} 
  $$
  \Delta = a_1^2-4a_0a_2
  $$
  and the latter is indeed invariant under the action of $G$.
  Moreover, it is in some sense the only invariant of $G$:
  if $\alg{A}=\Pol(\IR_2[x])$ is the algebra of polynomials over
  $\IR_2[x]$ (i.e., the polynomials in the coefficients $a_0$, $a_1$, $a_2$)
  then 
  $\alg{A}^G$ is the subalgebra generated by $\Delta$ and $\Delta$ is 
  the only ``simple'' invariant.
\end{Example}
Returning to the general case, 
if $\alg{A}$ is graded 
$$
\alg{A}=\bigoplus_{n\geq 0} \alg{A}_n
$$
with $\dim \alg{A}_n<\infty$ then one is interested in the dimensions
$d_n=\dim\alg{A}_n^G$. These are collected in the \emph{Hilbert-Poincar\'e series}
$$
H(\alg{A}^G;z) = \sum_{n=0}^\infty d_n z^n
$$

\subsection{Notation}
Before proceeding to the invariants of interest let us fix some notation.
We are going to consider matrix groups with their actions on certain vector spaces.
Let $V$ be a (complex) vector space. As usual $V^*$ denotes
the space of linear functionals $v^*:V\to \IC$ and there is a natural
dual pairing $\langle v^*,w\rangle = v^*(w)$ on $V^*\times V$.

The standard action of $G$ on $V=\IC^2$ induces a dual action on $V^*$ 
via \eqref{eq:inducedaction}, namely
$$
\langle g\cdot v^*,v\rangle = \langle v^*, g^{-1} \cdot v\rangle 
$$
i.e., by the invariance requirement 
$$
\langle g\cdot v^*, g\cdot v\rangle =\langle v^*,v\rangle
.
$$
Next we induce the action on $V^{*m}\times V^n$ by setting
$$
g\cdot (v_1^*,v_2^*,\dots,v_m^*,w_1,w_2,\dots,w_n) = 
 (g\cdot v_1^*,g\cdot v_2^*,\dots,g\cdot v_m^*,g\cdot w_1,g\cdot w_2,\dots,g\cdot w_n)
$$
Then for example, on $V^*\times V$, the map
\begin{equation}
  \label{eq:fVxVtoC}
\begin{aligned}
f:V^*\times V&\to \IC\\
(v^*,w) &\mapsto \langle v^*,w\rangle 
\end{aligned}
\end{equation}
is invariant under the action and similarly for any $m\in\IN$  the map
\begin{align*}
  V^{*m}\times V^m &\to \IC\\
  (v_1^*,v_2^*,\dots,v_m^*,w_1,w_2,\dots,w_m)
  &\mapsto \langle v_1^*,w_1\rangle \langle v_2^*,w_2\rangle \dotsm\langle v_m^*,w_m\rangle 
\end{align*}
In fact, for $V=\IC^n$ these are the only multilinear functions
which are invariant under the canonical action of
$\GL(n,\IC)$, but for $\SL(n,\IC)$ there are more as we shall see below.

The space of $d$-linear functionals $f:V^d\to \IC$ will be identified
with the $d$-fold tensor product $T^d(V^*)=V^*\otimes\dots\otimes V^*$.
We denote by $S^d(V^*)$ the subspace of symmetric $d$-linear forms,
i.e., the $d$-linear functionals which are invariant under permutation
of the arguments. This space can be identified with the space
of $d$-homogeneous polynomials on $V$ as follows.
First note that a symmetric $d$-linear form $f$ is completely determined
by the values of the $d$-homogeneous map $\tilde{f}(v)=f(v,v,\dots,v)$, 
because the other values can be obtained by \emph{polarization}:
$$
f(v_1,v_2,\dots,v_d) = 
\sum_{I\subseteq \{1,\dots,d\}} (-1)^{d-\abs{I}} \tilde{f}(\sum_{i\in I} v_i)
.
$$
Now if we choose a basis $e_1,\dots,e_n$ of $V$ and denote
$x_1,\dots,x_n$ the dual basis of $V^*$, 
then $S^d(V^*)$ is the linear span
of the monomials $x_{k_1,\dots,k_n}=x_1^{k_1}x_2^{k_2}\dots x_n^{k_n}$ with
$\sum k_i=d$ where 
$$
\tilde{x}_{k_1,\dots,k_n}(v) = \langle x_1,v \rangle^{k_1} \dotsm \langle x_n,v \rangle^{k_n} 
$$

\subsection{Classical Invariant Theory}

In the present paper we are interested in certain invariants of
$G=\SL(2,\IC)$, which acts on $V=\IC^2$ by left multiplication.
\emph{Classical invariant theory} is interested in the invariants of
the space $\RR_d$ of $d$-homogeneous polynomials on $V$,
which are called \emph{binary forms} of degree $d$.
Denoting the standard basis vectors of the dual space $V^*$ by $X$ and $Y$,
this space can be written as
$$
\RR_d=
\bigl\{
  \sum_{k=0}^d \binom{d}{k} \xi_k X^kY^{d-k} : \xi_k\in \IC
\bigr\}
$$
and it is isomorphic to $S^d(V^*)$, the $d$-fold symmetric tensor product
of $V^*$. The object of classical invariant theory are the 
polynomials in the coefficients $\xi_0,\xi_1,\dots,\xi_d$ which are
invariant under the action of $\SL(2,\IC)$, i.e., the space
$$
\bigoplus_{m\geq0} S^m(\RR_d^*)^{\SL(2,\IC)}
.
$$
Similarly, \emph{noncommutative invariant theory} is interested
in the invariant noncommutative polynomials, i.e., the
invariant elements of the full tensor algebra
$$
\bigoplus_{m\geq0} T^m(\RR_d^*)
.
$$
Indeed the $m$-fold tensor product $T^m(\RR_d^*)$ can be identified
with the space of $m$-linear forms on $\RR_d$ as follows:
Denote by $a_0,a_1,\dots,a_d$ the canonical basis of $\RR_d^*$,
i.e.,
$$
\langle
  a_k, \sum_{j=0}^d \binom{d}{j} \xi_j X^j Y^{d-j}
\rangle 
 = \xi_k
$$
Then the space of $m$-linear forms 
$T^m(\RR_d^*)$ is spanned by the non-commuting monomials
\begin{equation}
\label{eq:defmonomials}
\begin{split}
a_{k_1}a_{k_2}\dotsm a_{k_m} 
\biggl( 
  \sum_{j=0}^d \binom{d}{j} \xi_{1j} X^j Y^{d-j},
  \sum_{j=0}^d \binom{d}{j} \xi_{2j} X^j Y^{d-j},
  \dots,
  \sum_{j=0}^d \binom{d}{j} \xi_{mj} X^j Y^{d-j}
\biggr) \\
= \xi_{1k_1}\xi_{2k_2}\dotsm \xi_{mk_m}
\end{split}
\end{equation}
and we want to determine the space $T^m(\RR_d^*)^G$ 
of noncommutative polynomials
which are invariant under the action of $G=\SL(2,\IC)$ on $\RR_d$.

\subsection{The fundamental theorems}

Let us now take a closer look at the actions of $G=\SL(2,\IC)$ on $V=\IC^2$ 
and its dual.
There are more invariant
functions than for $\GL(2,\IC)$.
Denoting the standard basis vectors of $V$ by $e_1$ and $e_2$
and decomposing $v_i=\eta_{i1}e_1+\eta_{i2}e_2$ we can define 
another invariant function, namely the \emph{bracket} 
\begin{align*}
  V\times V&\to \IC\\
  (v_1,v_2)&\mapsto \NCsymbol{v_1}{v_2}:= \det
  \begin{bmatrix}
    \eta_{11} & \eta_{21}\\
    \eta_{12} & \eta_{22}
  \end{bmatrix}
\end{align*}
This function is indeed invariant, because
\begin{align*}
  \NCsymbol{g\cdot v_1}{g\cdot v_2}
  &= \det
     \left(
       g\cdot 
       \begin{bmatrix}
         \eta_{11} & \eta_{21}\\
         \eta_{12} & \eta_{22}
       \end{bmatrix}
     \right) \\
  &= \det g \,\det 
     \begin{bmatrix}
       \eta_{11} & \eta_{21}\\
       \eta_{12} & \eta_{22}
     \end{bmatrix}
\\
  &= \NCsymbol{v_1}{v_2}
\end{align*}
Similarly one can define a determinant on $V^*\times V^*$.
The first fundamental theorem states that
these together with \eqref{eq:fVxVtoC}
are \emph{all} the invariant functions.
\begin{Theorem}[First Fundamental Theorem]
  Every $\SL(2,\IC)$-invariant multilinear function
  $f:V^{*m}\times V^n\to \IC$ is a linear combination of products of
  the functions
  \begin{equation}
    \label{eq:invariants:generators}
    \langle v^*,w\rangle 
    \qquad \NCsymbol{v_1^*}{v_2^*} 
    \qquad \NCsymbol{w_1}{w_2}
  \end{equation}
\end{Theorem}
The functions~\eqref{eq:invariants:generators} are not independent from each
other, they satisfy certain relations, called \emph{syzygies}:
\begin{align}
  \label{eq:invariants:syzygies1}
  \NCsymbol{v_1}{v_2} &= -\NCsymbol{v_2}{v_1} \\
  \label{eq:invariants:syzygies2}
  \NCsymbol{v_1}{v_2} \NCsymbol{v_3}{v_4} &= \NCsymbol{v_1}{v_3} \NCsymbol{v_2}{v_4} - \NCsymbol{v_1}{v_4}\NCsymbol{v_2}{v_3}
\end{align}
The identity~\eqref{eq:invariants:syzygies2} is called
\emph{Pl\"ucker relation}.
The second fundamental theorem states that these are the only relations.
\begin{Theorem}[Second Fundamental Theorem]
  The algebra of invariant polynomials on $V^{*m}\times V^n$
  is isomorphic to the free algebra generated by the
  functions~\eqref{eq:invariants:generators} modulo the relations 
  \eqref{eq:invariants:syzygies1} and 
  \eqref{eq:invariants:syzygies2} (and their analogs on $V^*\times V^*$).
\end{Theorem}
We have thus a complete classification of the invariant functions
on $V^{*m}\times V^n$ and the so-called \emph{symbolic method}
provides a means to reduce other spaces to this one.

\section{Free Probability}
\label{sec:freeprobability}
Free probability was invented by Voiculescu~\cite{Voiculescu:1985:symmetries}
as a means to study the von Neumann algebras of free groups,
see~\cite{VoiculescuDykemaNica:1992:free}.
For our purpose the combinatorial approach of R.~Speicher is appropriate,
see the lectures~\cite{NicaSpeicher:2006:lectures}
for information beyond the following short survey.

The basic notion of free probability is a
\emph{noncommutative probability space}
$(\alg{A},\phi)$ which consists of a unital $C^*$-algebra $\alg{A}$ and a faithful state
$\phi$ (i.e., a linear functional $\phi:\alg{A}\to\IC$ with the properties $\phi(I)=1$ and
$\phi(X^*X)\geq0$; faithfulness means that $\phi(X^*X)=0$ if and only if $X=0$).
The elements of $\alg{A}$ are called \emph{noncommutative random variables}.
This definition follows the general strategy of noncommutative geometry to replace
commutative algebras of functions by more general noncommutative ones.
In this case the commutative von Neumann algebra $L^\infty(\Omega,\alg{F},\mu)$ 
of bounded measurable functions associated to a probability space $(\Omega,\alg{F},\mu)$
provides the motivating example.
We call \emph{distribution} of a noncommutative random variable $X$ 
the collection of its moments
$$
\phi(X^{k_1}X^{*k_1}X^{k_2}X^{*k_2}\dotsm X^{k_m}X^{*k_m})
.
$$
When considering a bounded selfadjoint random variable $X$, the sequence
of moments $\phi(X^k)$, $k=1,2,\dots$ uniquely determines a probability
measure $\mu_X$ on the spectrum of $X$, which is called the
\emph{(spectral) distribution} of $X$ and satisfies
$$
\phi(X^k)=\int t^k\,d\mu_X(t)
$$
for all $k\in\IN$.
There are various notions of noncommutative independence, and
\emph{free independence} or \emph{freeness} is the most successful so far.
\begin{Definition}
  \label{def:freeness}
  Given a noncommutative probability space~$(\alg{A},\phi)$, the subalgebras
  $\alg{A}_i\subseteq \alg{A}$ are called \emph{free} if
  $$
  \phi(X_1X_2\dots X_n)=0
  $$
  whenever $X_j\in \alg{A}_{i_j}$ with $\phi(X_j)=0$ and $i_j\ne i_{j+1}$
  for $j=1,2,\dots,n-1$.
\end{Definition}
Free probability shares a lot of features from classical probability.
There is for example a central limit theorem which can be formulated
exactly like the classical one and the limit distribution is Wigner's
semicircle law
$$
d\mu(t) = \frac{1}{2\pi} \sqrt{4-t^2}\,dt
.
$$
The \emph{free convolution} of two measures $\mu$ and $\nu$,
denoted $\mu\boxplus\nu$, which is the distribution of the sum of
two free random variables $X$ and $Y$ with spectral distributions
$\mu_X=\mu$ and $\mu_Y=\nu$. 
This operation is well defined because 
it can be shown that the distribution
of $X+Y$ only depends on the distributions $\mu_X$ and $\nu_Y$
and not on the particular realizations of $X$ and $Y$.

Correspondingly, a probability measure $\mu$ is called \emph{free
infinite divisible} if for every $n$ there exists a measure
$\mu_n$, such that $\mu=\mu_n\boxplus\mu_n\boxplus\dotsm\boxplus\mu_n$
($n$-fold convolution). A random variable $X$ is called
free infinite divisible if its spectral distribution $\mu_X$ has
this property.

To compute the free convolution,
the r\^ole of the \emph{characteristic function} of a random variable
is played by Voiculescu's \emph{$R$-transform}, but for our purposes
we chose Speicher's cumulant approach to freeness.

\subsection{Noncrossing Partitions}
\begin{Definition}
  Denote $\Pi_n$ the set of partitions
  $\pi=\{B_1,B_2,\dots,B_k\}$ 
  of the set $[n]=\{1,2,\dots,n\}$.
  Equivalently, a partition $\pi$ can be defined by
  the equivalence relation $\sim_\pi$ on $[n]$ whose
  equivalence classes are the blocks $B_j$ of $\pi$, i.e.,
  $$
  i \sim_\pi j 
  \iff
  \text{$i$ and $j$ belong to the same block of $\pi$}
  .
  $$
  A \emph{crossing} of $\pi$ is a quadruple $i<i'<j<j'$
  such that $i\sim_\pi j$, $i'\sim_\pi j'$ and $i\not\sim_\pi i'$.
  A partition $\pi$ is called 
  \emph{noncrossing} if it has no crossings.
  We represent partitions by diagrams as shown in Fig.~\ref{fig:partitions}.
  \begin{figure}
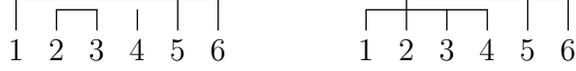

    \centering
\begin{asy}
  size(0.3*textwidth);

  label("1",(1,0),S);
  label("2",(2,0),S);
  label("3",(3,0),S);
  label("4",(4,0),S);
  label("5",(5,0),S);
  label("6",(6,0),S);

  draw((1,0)--(1,h1));
  draw((2,0)--(2,h0));
  draw((3,0)--(3,h0));
  draw((4,0)--(4,h0));
  draw((5,0)--(5,h1));
  draw((6,0)--(6,h1));

  draw((1,h1)--(6,h1));

  draw((2,h0)--(3,h0));
\end{asy}
\qquad
\begin{asy}
  size(0.3*textwidth);

  label("1",(1,0),S);
  label("2",(2,0),S);
  label("3",(3,0),S);
  label("4",(4,0),S);
  label("5",(5,0),S);
  label("6",(6,0),S);

  draw((1,0)--(1,h0));
  draw((2,0)--(2,h1));
  draw((3,0)--(3,h0));
  draw((4,0)--(4,h0));
  draw((5,0)--(5,h1));
  draw((6,0)--(6,h1));

  draw((2,h1)--(6,h1));

  draw((1,h0)--(4,h0));
\end{asy}
    \caption{The partitions $\{\{1,5,6\},\{2,3\},\{4\} \}$ and $\{\{1,3,4\},\{2,5,6\}\}$}
    \label{fig:partitions}
  \end{figure}
  Thus a partition is noncrossing if and only if its diagram can be drawn
  with no intersecting lines.
  We denote $\NC(n)$ the set of noncrossing partitions of the $n$-element set $[n]$.
  Equipped with the refinement order it is a lattice with minimal element
  $\hat{0}_n = \{\{1\},\{2\},\dots,\{n\}\}$
  and maximal element
  $\hat{1}_n = \{\{1,2,\dots,n\}\}$.
  We denote the lattice operations as usual $\pi\wedge\rho$ 
  and $\pi\vee\rho$.
  For a function $h:[n]\to A$ where $A$ is an arbitrary set, we denote
  $\ker h$ the partition of $[n]$ induced by the level sets of $h$, i.e.,
  the equivalence relation $i\sim j\iff h(i)=h(j)$.
\end{Definition}
Noncrossing partitions are enumerated by the ubiquitous \emph{Catalan numbers}
$$
\abs{\NC(n)} = C_n = \frac{1}{n+1}\binom{2n}{n}
.
$$
The M\"obius function is given by Catalan numbers as well,
$$
\mu(\hat{0}_n,\hat{1}_n) = (-1)^{n-1} C_{n-1}
.
$$

\subsection{Free Cumulants}
Given a noncrossing partition $\pi=\{B_1,B_2,\dots,B_k\}\in \NC(n)$
and random variables $X_1,X_2,\dots,X_n$ in some noncommutative 
probability space $(\alg{A},\phi)$, we define the partitioned
expectation
\begin{equation}
  \label{eq:partitionedexpectation}
  \phi_\pi(X_1,X_2,\dots,X_n) = \prod_{B\in\pi} \phi_B(X_1,X_2,\dots,X_n)
\end{equation}
where for a subset $B\subseteq [n]$ we denote the ordered partial moments
$$
\phi_B(X_1,X_2,\dots,X_n)  = \phi(\overrightarrow{\prod_{i\in B}} X_i)
$$
Following Speicher
\cite{Speicher:1994:multiplicative,NicaSpeicher:2006:lecture} we define 
the \emph{free cumulants} $C_n$ by the requirement
$$
\phi_\pi(X_1,X_2,\dots,X_n) = \sum_{\substack{\sigma\in\NC(n) \\
    \sigma\leq\pi}} C_\pi(X_1,X_2,\dots,X_n)
$$
where
$$
C_\pi(X_1,X_2,\dots,X_n) = \prod_{B\in\pi} C_B(X_1,X_2,\dots,X_n)
$$
similar to~\eqref{eq:partitionedexpectation}.
If we consider a single random variable, we write
$$
C_n(X) = C_n(X,X,\dots,X)
.
$$
By M\"obius inversion, this is equivalent to defining
$$
C_\pi(X_1,X_2,\dots,X_n) = \sum_{\substack{\sigma\in\NC(n) \\
    \sigma\leq\pi}} \phi_\pi(X_1,X_2,\dots,X_n)\,\mu(\sigma,\pi)
$$
Speicher~\cite{Speicher:1994:multiplicative} discovered that freeness is equivalent to the vanishing
of mixed cumulants, i.e.,
in the notation of Definition~\ref{def:freeness}
$$
C_n(X_1,X_2,\dots,X_n) = 0
$$
whenever $X_j\in \alg{A}_{i_j}$, $n\geq2$ and at least two $i_j$ are different;
see also \cite{Lehner:2004:cumulants1} 
for an explanation why noncrossing partitions appear.
We are going to deal with identically distributed free random variables
and apply the above formalism in the following situation.
\begin{Corollary}
  Let $(X_i)_{i\in\IN}$ be identically distributed free copies 
  of a random variable $X$ 
  from a noncommutative probability space $(\alg{A},\phi)$
  and $h:[n]\to \IN$ an index map. 
  Then
  $$
  \phi(X_{h(1)}X_{h(2)}\dots X_{h(n)}) 
  = \sum_{\pi \leq \ker h} C_\pi(X)
  .
  $$
\end{Corollary}
For example,
the only nonvanishing free cumulant of Wigner's semicircle law is $c_2$
and it follows that in the normalized case where $c_2=1$ 
the $2n$-th moment equals the number of noncrossing pair partitions
on $2n$ elements. Similarly the normalized \emph{free Poisson law} is 
characterized by the property that all free cumulants $c_n=1$ and thus
the $n$-th moment equals the number of noncrossing partitions on $n$
elements, that is, again the Catalan numbers.

\section{A Noncrossing Basis for Noncommutative Invariants}
\label{sec:noncrossing}
\subsection{The Symbolic Method \cite{Tambour:1991:noncommutative}}

We look for invariants of $T^m(\RR_d^*)$,
the space of $m$-homogeneous polynomials in the  noncommuting variables
$a_0$, $a_1$,\ldots,$a_d$,
under the induced action of $SL(2,\IC)$.
Earlier~\eqref{eq:defmonomials} we have identified 
these with $m$-linear forms on $R_d=S^d(V^*)$ and now in order to apply
the fundamental theorems we have to relate these to invariants of
$V^{*k}\times V^l$ for some $k$ and $l$. This is accomplished by 
Tambour's 
\emph{Symbolic Method}~\cite{Tambour:1991:noncommutative}
which proceeds as follows.
Denote $\symm:T^d(V^*)\to S^d(V^*)$ the projection (``symmetrizator'') which
maps a $d$-linear form on $V$ to its symmetrization.
Now every $m$-linear form $F$ on $\RR_d=S^d(V^*)$ is an
element of the tensor space $T^m(\RR_d^*)$ and extends to an $m$-linear
form $\symm^*F$ on $T^d(V^*)$, i.e., an element of $T^m(T^d(V^*))$ by setting,
for $z_1,z_2,\dots,z_n\in T^d(V^*)$,
$$
\symm^*F(z_1,z_2,\dots,z_m) = F(\symm (z_1),\symm(z_2),\dots,\symm(z_m))
$$
and a fortiori an $md$-linear form $\omega_F$ on $V^*$, called the
\emph{symbol}, namely
\begin{multline*}
  \omega_F(y_{11},y_{12},\dots,y_{1d},y_{21},\dots,y_{2d},\dots,y_{m1},\dots,y_{md})\\
  \begin{aligned}
    &= \symm^*F(y_{11}\otimes y_{12}\otimes\dots\otimes y_{1d},
                y_{21}\otimes\dots\otimes y_{2d},
                \dots, 
                y_{m1}\otimes\dots\otimes y_{md}) \\
    &= F(\symm(y_{11}\otimes y_{12}\otimes\dots\otimes y_{1d}),
         \symm(y_{21}\otimes\dots\otimes y_{2d}),
         \dots, 
         \symm(y_{m1}\otimes\dots\otimes y_{md}))
  .
  \end{aligned}
\end{multline*}
Now it is immediate that the symbol $\omega_F$ is invariant under
permutation of each block of arguments $y_{i1},y_{i2},\dots,y_{id}$
and because of this symmetry it will be enough to consider
 map
$$
\tilde{\omega}_F:
(y_1,y_2,\dots,y_m)\mapsto\omega_F(y_1,\dots,y_1,y_2,\dots,y_2,\dots,y_m,\dots,y_m)
$$
which is $d$-homogeneous in each variable.
\begin{Example}
  Consider the linear functional $a_0\in T^1(R_d^*)$ which is defined by
  $$
  a_0(\sum\binom{d}{k} \xi_k X^k Y^{d-k}) = \xi_0
  .
  $$
  Its symbol evaluated at $y=\eta_1 X+\eta_2Y\in V^*$ is
  \begin{align*}
    \omega_{a_0}(y,\dots,y) 
    &= \symm^*a_0(y\otimes y\otimes \dots \otimes y)\\
    &= a_0((\eta_1X+\eta_2Y)^d)\\
    &= a_0(\sum \binom{d}{k} \eta_1^k\eta_2^{d-k} X^k Y^{d-k})\\
    &= \eta_2^d
  \end{align*}
\end{Example}
So far we have shown one half of the following lemma.
\begin{Lemma}
  The invariant $m$-linear forms on $R_d$ are in one-to-one correspondence
  with the $md$-linear forms on $V^*$ which are invariant under
  permutations $\SG_d\times \SG_d\times\dots\times \SG_d$.
\end{Lemma}
The opposite process which reconstructs an $m$-linear
form from its symbol is called \emph{restitution}
and establishes the other half of the lemma.
We start  with an example.
\begin{Example}
The $2d$-linear $\SG_d\times\SG_d$-invariant map  corresponding to the
multi-$d$-homogeneous
map
\begin{align*}
  \omega:(V^{*})^2&\to \IC\\
  (y_1,y_2) &\mapsto \NCsymbol{y_1}{y_2}^d
\end{align*}
is invariant and is the symbol of the following $d$-linear form:
\begin{align*}
  \NCsymbol{y_1}{y_2}^d
  &= \det
     \begin{pmatrix}
       \eta_{11} & \eta_{21}\\
       \eta_{12} & \eta_{22}
     \end{pmatrix}^d \\
  &= (\eta_{11}\eta_{22}-\eta_{12}\eta_{21})^d\\
  &= \sum_{k=0}^d \binom{d}{k} (-1)^k (\eta_{11}\eta_{22})^{d-k}(\eta_{12}\eta_{21})^k \\
  &= \sum_{k=0}^d \binom{d}{k} (-1)^k \eta_{11}^{d-k}  \eta_{12}^k\eta_{21}^k\eta_{22}^{d-k} \\
\end{align*}
This means that
$$
F(\sum\binom{d}{k} \xi_{1k} X^k Y^{d-k},\sum\binom{d}{k} \xi_{2k} X^k Y^{d-k}) 
= \sum_{k=0}^d \binom{d}{k} (-1)^k \xi_{1,d-k}\xi_{2,k}
$$
i.e.,
$$
F=\sum_{k=0}^d \binom{d}{k} (-1)^k a_{d-k}a_k
$$
and for $d=2$ this is the noncommutative discriminant $a_2a_0-a_1a_1+a_0a_2$.
\end{Example}
In general, if $\omega:(V^*)^{md}$ is an invariant multilinear functional
which is also invariant under permutations from $\SG_d^m$,
then by the first fundamental theorem the value
$$\omega(y_1,y_1,\dots,y_1,
        y_2,y_2,\dots,y_2,
        \dots
        y_m,y_m,\dots,y_m)
$$
must be a linear combination
of products of brackets 
$\NCsymbol{y_i}{y_j}$ with $i\ne j$ where each $y_i$ appears
exactly $d$ times. Thus by linearity it suffices to construct for each
$(m-d)$-homogeneous form
$$
\tilde{\omega}(y_1,y_2,\dots,y_m) = \prod \NCsymbol{y_{i_k}}{y_{j_k}}
$$
satisfying the condition just stated
a noncommutative invariant whose symbol is $\omega$.
Now if we decompose $y_i=\eta_{i1} X + \eta_{i2}Y$ we have
$$
\tilde{\omega}(y_1,y_2,\dots,y_m) = \prod (\eta_{i_k1} \eta_{j_k2} -
\eta_{i_k2} \eta_{j_k1} )
$$
and expanding the product we get a sum of terms of the form
$$
\prod_{i=1}^m \eta_{i1}^{s_i} \eta_{i2}^{d-s_i}
$$
which is the symbol of the noncommutative monomial
$$
\prod_{i=1}^m a_{s_i}
.
$$

\subsection{Finding a basis}
We have thus used the first fundamental theorem to determine all invariants;
namely, the symbols are spanned by the elementary symbols
$$
\prod \NCsymbol{y_i}{y_j}
$$
where each $y_i$ appears exactly $d$ times. 
For finding a basis it is convenient to use diagrams.
\begin{Definition}
  \label{def:mpartite}
  An \emph{$m$-partite partition} of the set $[dm]$ is a partition
  whose blocks contain at most one element from each interval
  $\{kd+1,kd+2,\dots,(k+1)d\}$. 
  To each $m$-partite pair partition
  $\pi=\{\{i_1,j_1\},\dots,\{i_q,j_q\}\}$
  we associate the symbol
  $$
  \tilde{\omega}_\pi(y_1,\dots,y_m) = 
  \prod \NCsymbol{y_{i_k}}{y_{j_k}}
  $$
\end{Definition}

It is easy to see that different partitions may lead to identical symbols
and in particular the corresponding symbols are not linearly independent.
Moreover the Pl\"ucker relations lead to even more linear dependencies.
We shall show that the latter is true if we restrict to
noncrossing $m$-partite pair partitions.
Moreover in the rest of this section we prove that they form a basis:
\begin{Theorem}
  \label{thm:dimTmRRdG}
  The dimension of the space
  $T^m(\RR_d^*)^G$ of invariant noncommutative polynomials
  is equal to the number of $m$-partite noncrossing pair partitions
  $\pi\in\NC(md)$.
\end{Theorem}

The key observation is that the Pl\"ucker relation
$$
\NCsymbol{v_1}{v_3}\NCsymbol{v_2}{v_4} = \NCsymbol{v_1}{v_2}\NCsymbol{v_3}{v_4} + \NCsymbol{v_1}{v_4}\NCsymbol{v_2}{v_3} 
$$
has a pictorial interpretation as follows:
$$
\begin{asy}
  size(0.15*textwidth);
  draw((1,0)--(1,h0));
  draw((2,0)--(2,h1));
  draw((3,0)--(3,h0));
  draw((4,0)--(4,h1));
  draw((1,h0)--(3,h0));
  draw((2,h1)--(4,h1));
\end{asy}
=
\begin{asy}
  size(0.15*textwidth);
  draw((1,0)--(1,h0));
  draw((2,0)--(2,h0));
  draw((3,0)--(3,h0));
  draw((4,0)--(4,h0));
  draw((1,h0)--(2,h0));
  draw((3,h0)--(4,h0));
\end{asy}
+
\begin{asy}
  size(0.15*textwidth);
  draw((1,0)--(1,h1));
  draw((2,0)--(2,h0));
  draw((3,0)--(3,h0));
  draw((4,0)--(4,h1));
  draw((1,h1)--(4,h1));
  draw((2,h0)--(3,h0));
\end{asy}
$$

We see that the number of crossings is reduced by one and
this means that if we start with an elementary symbol
$$
\prod \NCsymbol{y_{i_k}}{y_{j_k}}
$$
we can associate to it a pairing
and successively remove any crossings to obtain a linear combination
of noncrossing pairings. Thus the space of symbols is spanned by 
noncrossing symbols.
This strategy is different from the usual
\emph{straightening algorithm} where the
formula is read as
$$
\NCsymbol{y_1}{y_4}\NCsymbol{y_2}{y_3} = \NCsymbol{y_1}{y_3}\NCsymbol{y_2}{y_4}-\NCsymbol{y_1}{y_2}\NCsymbol{y_3}{y_4}
,
$$

$$
\begin{asy}
  size(0.15*textwidth);
  draw((1,0)--(1,h1));
  draw((2,0)--(2,h0));
  draw((3,0)--(3,h0));
  draw((4,0)--(4,h1));
  draw((1,h1)--(4,h1));
  draw((2,h0)--(3,h0));
\end{asy}
=
\begin{asy}
  size(0.15*textwidth);
  draw((1,0)--(1,h0));
  draw((2,0)--(2,h1));
  draw((3,0)--(3,h0));
  draw((4,0)--(4,h1));
  draw((1,h0)--(3,h0));
  draw((2,h1)--(4,h1));
\end{asy}
-
\begin{asy}
  size(0.15*textwidth);
  draw((1,0)--(1,h0));
  draw((2,0)--(2,h0));
  draw((3,0)--(3,h0));
  draw((4,0)--(4,h0));
  draw((1,h0)--(2,h0));
  draw((3,h0)--(4,h0));
\end{asy}
$$
i.e., nestings are removed.
The straightening algorithm has the advantage
to be applicable for arbitrary $\SL(n,\IC)$,
whereas our approach only works for $\SL(2,\IC)$.
The next lemma concludes the proof of Theorem~\ref{thm:dimTmRRdG}.
\begin{Lemma}
  Symbols coming from different noncrossing pairings are linearly independent.
  The irreducible noncrossing pairings, that is, those in which 
  the left- and rightmost vertices are connected with each other,
  generate the invariants as a ring.
\end{Lemma}
\begin{proof}
  This can be shown as in \cite{Teranishi:1988:noncommutative}.
  We order the noncommutative monomials in $T^m(R_d^*)$ lexicographically
  with respect to the order
  $a_{d}>a_{d-1}>\dots> a_{0}$ on the letters
  and we will show that different noncrossing symbols
  have different leading terms with respect to this order.
  Let us first consider an example:
$$
  \begin{aligned}
\begin{asy}
  size(0.4*textwidth);
  draw((1,0)--(1,h3));
  draw((2,0)--(2,h2));
  draw((3,0)--(3,h1));
  draw((4,0)--(4,h0));
  draw((5,0)--(5,h0));
  draw((6,0)--(6,h1));
  draw((7,0)--(7,h1));
  draw((8,0)--(8,h0));
  draw((9,0)--(9,h0));
  draw((10,0)--(10,h1));
  draw((11,0)--(11,h2));
  draw((12,0)--(12,h3));
  draw((1,h3)--(12,h3));
  draw((2,h2)--(11,h2));
  draw((3,h1)--(6,h1));
  draw((4,h0)--(5,h0));
  draw((7,h1)--(10,h1));
  draw((8,h0)--(9,h0));
  label("1",(1,0),S);
  label("1",(2,0),S);
  label("1",(3,0),S);
  label("1",(4,0),S);
  label("2",(5,0),S);
  label("2",(6,0),S);
  label("2",(7,0),S);
  label("2",(8,0),S);
  label("3",(9,0),S);
  label("3",(10,0),S);
  label("3",(11,0),S);
  label("3",(12,0),S);
\end{asy}
  &= \NCsymbol{1}{2}^2\NCsymbol{1}{3}^2\NCsymbol{2}{3}^2 \\
  &= (\eta_{11}\eta_{22}-\eta_{12}\eta_{21})^2
     (\eta_{11}\eta_{32}-\eta_{12}\eta_{31})^2
     (\eta_{21}\eta_{32}-\eta_{22}\eta_{31})^2 \\
  &= \eta_{11}^4\eta_{21}^2\eta_{22}^2\eta_{32}^4+\dotsm\\
  &\simeq a_4a_2a_0 + \dotsm
\end{aligned}
$$
  Consider the $k$-th interval $\{(k-1)d+1,(k-1)d+2,\dots,kd\}$.
  An edge adjacent to this interval is called \emph{incoming} 
  if it connects to an element to the left and
  \emph{outgoing} if it connects to the right.
  Then the index of the $k$-th factor of the leading term
  indicates the number of outgoing edges of the $k$-th interval.

  Since in a noncrossing partition the incoming edges always come 
  before the outgoing edges, these numbers uniquely determine the partition.
  Thus different noncrossing partitions have different leading terms.

  As in \cite{Teranishi:1988:noncommutative} one can show that the invariants
  coming from noncrossing irreducible symbols
  (i.e., those with only one outer block)
  form a free generating set of the ring of noncommutative invariants.
\end{proof}

Note that this also yields an explicit bijection between $m$-partite
noncrossing pair partitions and column-strict Young tableaux.
This combinatorial coincidence was also found independently 
in~\cite{BenaychNechita:2008:permutation} by establishing a bijection with
the Young tableaux of Teranishi~\cite{Teranishi:1988:noncommutative}.

\section{Free Stochastic Measures and the Hilbert series}
\label{sec:freestochastic}
In order to find the Hilbert-Poincar\'e series
\begin{equation}
\label{eq:hilbertseries}
H_d(z)=
H(T(\RR_d^*)^G;z)
= \sum_{m=0}^\infty \dim T^m(\RR_d^*)^G\, z^m
\end{equation}
for fixed $d$ one usually resorts to integration on the
group (Molien's formula)
which in our case reads
\begin{Theorem}[{ \cite{AlmkvistDicksFormanek:1985:Hilbert}}]
  \label{thm:almkvist}
  $$
  H_d(z)
  = \frac{2}{\pi} 
    \int_0^\pi
    \frac{\sin^2 x}%
    {1-\frac{\sin(d+1)x}%
            {\sin x}\,z}
    \,
    dx
  $$
\end{Theorem}
Our aim here is to provide a different proof of this
by establishing a combinatorial link to \emph{free stochastic 
measures}. The latter have been constructed by
Anshelevich~\cite{Anshelevich:2000:freestochastic} following
Rota and Wallstrom~\cite{RotaWallstrom:1997:stochastic}.
Let $X$ be a free infinitely divisible random variable.
Then for every $N\in\IN$ we can write $X$ as a sum
of identically distributed free random variables $X_i^{(N)}$,
$i\in\{1,\dots,N\}$ and for 
every partition $\pi\in\Pi_n$ the \emph{stochastic measure} $\StochMeas_\pi$
and the \emph{product measure} $\ProdMeas_\pi$ are defined as the elements
\begin{align*}
\StochMeas_\pi 
&= \lim_{N\to\infty} 
   \sum_{\ker h = \pi} 
   X_{h(1)}^{(N)}  X_{h(2)}^{(N)}\dotsm X_{h(n)}^{(N)}\\
\ProdMeas_\pi 
&= \lim_{N\to\infty} 
   \sum_{\ker h \geq \pi} 
   X_{h(1)}^{(N)}  X_{h(2)}^{(N)}\dotsm X_{h(n)}^{(N)}
   .
\end{align*}
It can be shown that the limits exist in norm
and
we will be particularly interested in the special cases
$\psi_n=\StochMeas_{\hat{0}_n}$ and the so called diagonal measures
$\Delta_n=\StochMeas_{\hat{1}_n}$. 
The following properties hold:
$\StochMeas_\pi=0$ unless $\pi$ is noncrossing
\cite[Thm.~1]{Anshelevich:2000:freestochastic}
and from this it follows immediately that
$$
\ProdMeas_\pi
= \sum_{\substack{\sigma\in\NC\\ \sigma\geq\pi}}
  \StochMeas_\sigma
\qquad
\StochMeas_\pi
= \sum_{\substack{\sigma\in\NC\\ \sigma\geq\pi}}
  \mu(\pi,\sigma)\,\ProdMeas_\sigma
$$
Moreover, by \cite[Lemma~1]{Anshelevich:2000:freestochastic},
the expectation of a stochastic measure has
a simple expression in terms of cumulants of the original random variable $X$,
namely
$$
\phi(\StochMeas_\pi)=C_\pi(X)
.
$$
Concerning the joint distribution of $\psi_n$,
\cite[Prop.~4]{Anshelevich:2000:freestochastic}
tells us that
$$
\psi_{k_1}\psi_{k_2}\dotsm \psi_{k_m}
= \sum_{\substack{
          \sigma\in\NC(k_1+k_2+\dots+k_m)\\ 
          \sigma\wedge \hat{1}_{k_1}\hat{1}_{k_2}\dotsm\hat{1}_{k_m}=\hat{0}
  }}
  \StochMeas_\sigma
$$
where we recognize the $m$-partite partitions of Definition~\ref{def:mpartite}.
Altogether it follows that
\begin{equation}
  \label{eq:phipsik1...psikm}
  \phi(\psi_{k_1}\psi_{k_2}\dotsm \psi_{k_m})
  = \sum_{\substack{
                \sigma\in\NC(k_1+k_2+\dots+k_m)\\ 
                \sigma\wedge \hat{1}_{k_1}\hat{1}_{k_2}\dotsm\hat{1}_{k_m}=\hat{0}
              }}
     C_\sigma(X)
\end{equation}
An alternative inductive proof of this formula
is given in \cite{Mizuo:2003:noncommutative},
see also \cite{KempSpeicher:2007:strong} for an application to strong
Haagerup inequalities for so-called $R$-diagonal elements.

To conclude our proof of theorem~\ref{thm:almkvist}
let us from now on assume that $X$ is a standard semicircular element,
with $\phi(X^2)=1$.
Then
$$
C_\sigma(X)=
\begin{cases}
  1 & \sigma\in \NC_2\\
  0 & \sigma\not\in \NC_2
\end{cases}
$$
together with Theorem~\ref{thm:dimTmRRdG} implies that
\begin{equation}
  \label{eq:dimTm=phipsidm}
  \dim T^m(\RR_d^*)^G = \phi(\psi_d^m)
  .
\end{equation}
It remains to identify the distribution of $\psi_d$.
Here we use one more result of Anshelevich 
\cite[Prop.~5]{Anshelevich:2000:freestochastic}
which states that for a centered free infinite divisible
random variable we have the orthogonality relation
$$
\gamma(\psi_m\psi_n) = 
  \delta_{mn}\phi(\Delta_2)^n = \delta_{mn}C_2(X)^n
$$
and therefore $\psi_k$ can be identified
with the orthogonal polynomials of $X$,
which in the semicircular case are the Chebyshev
polynomials $U_n$ of the second kind and thus
\cite[Cor.~8]{Anshelevich:2000:freestochastic}
$$
\psi_n = X\psi_{n-1}-\psi_{n-2}
$$
i.e., $\psi_n=U_n(X)$.
Plugging this into \eqref{eq:dimTm=phipsidm} we obtain
\begin{align*}
  \dim T^m(\RR_d^*)^G
  &= \phi(U_d(X)^m)\\
  &= \int_{-2}^2 U_d(x)^m \sqrt{4-x^2}\,dx\\
  &= \frac{1}{\pi} 
     \int_{-\pi}^\pi
     \left(
       \frac{\sin(d+1)\theta}{\sin\theta}
     \right)^m
     \sin^2\theta\,d\theta
\end{align*}
by the standard substitution
$U_d(\cos \theta) = \frac{\sin(d+1)\theta}{\sin\theta}$.

\begin{Remark}
  If $d$ is even then the noncrossing $m$-partite
pair partitions are in bijection
with all $m$-partite noncrossing partitions without singletons on $md/2$
points via the thickening bijection illustrated in the following example:
$$
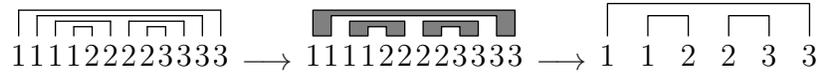
\begin{asy}
  size(0.3*textwidth);
  draw((1,0)--(1,h3));
  draw((2,0)--(2,h2));
  draw((3,0)--(3,h1));
  draw((4,0)--(4,h0));
  draw((5,0)--(5,h0));
  draw((6,0)--(6,h1));
  draw((7,0)--(7,h1));
  draw((8,0)--(8,h0));
  draw((9,0)--(9,h0));
  draw((10,0)--(10,h1));
  draw((11,0)--(11,h2));
  draw((12,0)--(12,h3));
  draw((1,h3)--(12,h3));
  draw((2,h2)--(11,h2));
  draw((3,h1)--(6,h1));
  draw((4,h0)--(5,h0));
  draw((7,h1)--(10,h1));
  draw((8,h0)--(9,h0));
  label("1",(1,0),S);
  label("1",(2,0),S);
  label("1",(3,0),S);
  label("1",(4,0),S);
  label("2",(5,0),S);
  label("2",(6,0),S);
  label("2",(7,0),S);
  label("2",(8,0),S);
  label("3",(9,0),S);
  label("3",(10,0),S);
  label("3",(11,0),S);
  label("3",(12,0),S);
\end{asy}
\longrightarrow
\begin{asy}
  size(0.3*textwidth);
  filldraw((1,0)--(1,h3)--(12,h3)--(12,0)--(11,0)--(11,h2)--(2,h2)--(2,0)--cycle,grey);

  filldraw((3,0)--(3,h1)--(6,h1)--(6,0)--(5,0)--(5,h0)--(4,h0)--(4,0)--cycle,grey);
  filldraw((7,0)--(7,h1)--(10,h1)--(10,0)--(9,0)--(9,h0)--(8,h0)--(8,0)--cycle,grey);

  label("1",(1,0),S);
  label("1",(2,0),S);
  label("1",(3,0),S);
  label("1",(4,0),S);
  label("2",(5,0),S);
  label("2",(6,0),S);
  label("2",(7,0),S);
  label("2",(8,0),S);
  label("3",(9,0),S);
  label("3",(10,0),S);
  label("3",(11,0),S);
  label("3",(12,0),S);
\end{asy}
\longrightarrow
\begin{asy}
  size(0.3*textwidth);
  draw((1,0)--(1,h1)--(6,h1)--(6,0));
  draw((2,0)--(2,h0)--(3,h0)--(3,0));  
  draw((4,0)--(4,h0)--(5,h0)--(5,0));
  label("1",(1,0),S);
  label("1",(2,0),S);
  label("2",(3,0),S);
  label("2",(4,0),S);
  label("3",(5,0),S);
  label("3",(6,0),S);
\end{asy}
$$
\end{Remark}


\emph{Acknowledgements.}
We thank Roland Speicher for bringing formula
\eqref{eq:phipsik1...psikm} to our attention.

\providecommand{\bysame}{\leavevmode\hbox to3em{\hrulefill}\thinspace}
\providecommand{\MR}{\relax\ifhmode\unskip\space\fi MR }
\providecommand{\MRhref}[2]{%
  \href{http://www.ams.org/mathscinet-getitem?mr=#1}{#2}
}
\providecommand{\href}[2]{#2}

\end{document}